\documentclass{article}
\usepackage[utf8]{inputenc}
\usepackage{amsmath , amsthm , amssymb}
\usepackage{float}
\usepackage{url}

\title{The Covering Numbers of the McLaughlin Group and some Primitive Groups of Low Degree}
\author{Michael Epstein}
\date{}

\newtheorem{lemma}{Lemma}[section]
\newtheorem{thm}[lemma]{Theorem}
\newtheorem*{thm*}{Theorem}

\newtheorem{prop}[lemma]{Proposition}

\begin{document}

\maketitle

\begin{abstract}
A \emph{finite cover} of a group $G$ is a finite collection $\mathcal{C}$ of proper subgroups of $G$ with the property that $\bigcup \mathcal{C} = G$. A finite group admits a finite cover if and only if it is noncyclic. More generally, it is known that a group admits a finite cover if and only if it has a finite, noncyclic homomorphic image. If $\mathcal{C}$ is a finite cover of a group $G$, and no cover of $G$ with fewer subgroups exists, then $\mathcal{C}$ is said to be a \emph{minimal cover} of $G$, and the cardinality of $\mathcal{C}$ is called the \emph{covering number} of $G$, denoted by $\sigma(G)$. Here we investigate the covering numbers of the McLaughlin sporadic simple group and some low degree primitive groups.
\end{abstract}

\section{Introduction}

A number of interesting results about covering numbers were proven in \cite{Cohn1994}, where it was conjectured that the covering number of a finite, noncyclic, solvable group is of the form $1+q$, where $q$ is the order of a chief factor of the group. This conjecture was proven by M. J. Tomkinson in \cite{Tomkinson1997}. In light of this result, much of the more recent work on covering numbers of finite groups has focused on nonsolvable groups, and in particular on simple and almost simple groups. A number of results can be found in \cite{Britnell2008, Britnell2011, Bryce1999, EMNP17, Holmes2006, Holmes2010, Kappe2010, Kappe2016, Maroti2005}. In this article, we investigate group covering numbers in some outstanding open cases from \cite{Garonzi2019} and \cite{Holmes2006}.

In Section \ref{McL}\footnotemark we determine the covering number of the McLaughlin sporadic simple group. This was first attempted by P. E. Holmes in \cite{Holmes2006}, where it was shown that $24541 \leq \sigma(McL) \leq 24553$. We will show that the upper bound of 24553 is in fact the correct covering number.

\footnotetext{Section \ref{McL} of this article is based on a portion of the author's dissertation \cite{Dissertation}, submitted in partial fulfillment of the requirements for the degree of doctor of philosophy at Florida Atlantic University.}

In \cite{Garonzi2019}, the authors investigated integers which are not covering numbers of groups, and in the process determined the covering numbers of many primitive groups of degree less than or equal to 129. However, a few difficult cases remain unsolved (see Table 2 of \cite{Garonzi2019} for upper and lower bounds). In Section \ref{Primitive} we consider some of these open cases. In particular, we find the exact covering numbers of $L_5(3)$, $P\Sigma L_2(121)$, $A_5\ \text{wr}\ 3$, $(A_5 \times A_5 \times A_5).6$, $P\Gamma L_2(125)$, and $L_7(2)$  and give improved bounds for the covering numbers of $A_7\ \text{wr}\ 2$, $L_3(4).2_2$\footnotemark, $HS:2$, $L_2(11)\ \text{wr}\ 2$, and $PGU_3(5)$.

\footnotetext{For clarity we use the same numbering for the groups $L_3(4).2$ as in \cite{Garonzi2019}. In \cite{Atlas1985}, this group is called $L_3(4).2_1$.}

\section{Preliminaries} \label{Prelim}

In this article we follow \cite{DandF} for basic group theoretic terminology and notation. We use the notation of \cite{Atlas1985} for simple groups and group structures.

The basic method we use for determining the covering numbers of finite groups is based on two simple observations: First, that one need only consider covers consisting of maximal subgroups, and second, that in order to prove that a collection $\mathcal{C}$ of proper subgroups is a cover of a group $G$ it is sufficient to show that each of the maximal cyclic subgroups is contained in a subgroup from $\mathcal{C}$. We call the maximal cyclic subgroups of $G$ the \emph{principal subgroups} of $G$ and a generator of a principal subgroup is called a \emph{principal element} of $G$. In light of these observations, the first step in determining the covering number of a group $G$ is to determine the maximal subgroups and principal subgroups of $G$ up to conjugacy.

The second step is to determine which principal subgroups are contained in which maximal subgroups. Consider a bipartite graph whose vertex set is the union of a conjugacy class $\mathcal{P}$ of principal subgroups of $G$ and a conjugacy class $\mathcal{M}$ of maximal subgroups of $G$, with an edge between $C \in \mathcal{P}$ and $H \in \mathcal{M}$ if and only if $C \leq H$. The group $G$ acts transitively on both $\mathcal{P}$ and $\mathcal{M}$ by conjugation, so any two vertices in $\mathcal{P}$ have the same degree $a_{\mathcal{P},\mathcal{M}}$ and any two vertices in $\mathcal{M}$ have the same degree $b_{\mathcal{P},\mathcal{M}}$. Moreover, counting the number of edges in the graph in two ways yields the equation
\begin{equation} \label{equation}
a_{\mathcal{P},\mathcal{M}} \lvert \mathcal{P} \rvert = b_{\mathcal{P},\mathcal{M}} \lvert \mathcal{M} \rvert.
\end{equation}
Consequently, one can construct two matrices $A = (a_{\mathcal{P},\mathcal{M}})$ and $B = (b_{\mathcal{P},\mathcal{M}})$, with rows indexed by the conjugacy classes of principal subgroups of $G$ and columns indexed by the classes of maximal subgroups, describing the incidence between the principal subgroups of $G$ and the maximal subgroups up to conjugacy. In this article we give only the matrix $A$ for each group as the entries tend to be smaller, and as noted above the entries of $B$ can easily be computed from those of $A$. In order to compute the entries of the matrices $A$ and $B$, we make use of the characters of the permutation representations of $G$ acting on the sets of right cosets of representatives of each conjugacy class of maximal subgroups. This is justified by the following well-known fact from character theory:

\begin{prop} \label{prop:char}
Let $H$ be a subgroup of $G$,  $x \in G$, $K$ be the conjugacy class of $x$, and $\theta$ be the permutation character of the action of $G$ on the right cosets of $H$. Then,
\[ \lvert K \cap H \rvert =  \dfrac{\theta(x)\lvert K \rvert}{\lvert G:H\rvert}\,.\]
\end{prop}

With this proposition we can compute the entries of $A$ and $B$ as follows:
\begin{prop} \label{prop:char2}
Let $x \in G$, $\mathcal{P}$ be the conjugacy class of $\langle x \rangle$, $H$ be a subgroup of $G$ from conjugacy class $\mathcal{M}$, and $\theta$ be the permutation character of the action of $G$ on the right cosets of $H$. Then
\[a_{\mathcal{P},\mathcal{M}} = \dfrac{\theta(x)\lvert \mathcal{M}\rvert}{\lvert G:H\rvert}\text{ and } b_{\mathcal{P},\mathcal{M}} = \dfrac{\theta(x)\lvert \mathcal{P}\rvert}{\lvert G:H\rvert}\,.\]
\end{prop}

\begin{proof}
Let $K$ be the conjugacy class of $x$. We consider another bipartite graph, with vertex set $K \cup \mathcal{M}$ and with an edge between $y \in K$ and $J \in \mathcal{M}$ if and only if $y \in J$. Any two vertices in $K$ have the same degree $\alpha$, any two vertices in $\mathcal{M}$ have the same degree $\beta$, and $\alpha \lvert K \rvert = \beta \lvert \mathcal{M} \rvert$. By Proposition \ref{prop:char}, $\beta = \frac{\theta(x)\lvert K \rvert}{\lvert G:H\rvert}$, and since $y \in J$ if and only if $\langle y \rangle \leq J$, $a_{\mathcal{P},\mathcal{M}} = \alpha = \frac{\beta \lvert \mathcal{M}\rvert}{\lvert K \rvert} = \frac{\theta(x)\lvert \mathcal{M}\rvert}{\lvert G:H\rvert}$. The formula for $b_{\mathcal{P},\mathcal{M}}$ then follows from \eqref{equation}.
\end{proof}

Sometimes it is possible to simplify the problem and reduce the number of conjugacy classes of subgroups we must consider. For example, if the members of some conjugacy class of maximal subgroups contain no principal subgroups, then this conjugacy class may be eliminated from further consideration. On the other hand, if a principal subgroup $C$ is contained in a unique maximal subgroup $H$, then the entire conjugacy class of $H$ must be used in the cover (as we only consider covers consisting of maximal subgroups). After performing such simplifications, we obtain a set $U$ (possibly empty) of maximal subgroups that we know we must use in a cover, and submatrices $A'$ and $B'$, of $A$ and $B$ respectively, which describe the incidence between the remaining classes of subgroups.
 
The third step is to find a small cover which will be a candidate for a minimal cover. This will often, but not always, be a union of conjugacy classes of maximal subgroups of $G$. One can easily find covers of this type; they consist of the subgroups from $U$ as well as those from the conjugacy classes of maximal subgroups corresponding to a set of columns of $A'$ whose sum has no zero entries. In any case, once a cover has been found, the number of subgroups in the cover gives an upper bound on $\sigma(G)$.

We establish a lower bound for the covering number by solving a certain integer linear programming (ILP) problem. To make this more explicit, consider a cover $\mathcal{C}$ consisting of the maximal subgroups of $G$, and let $\mathcal{M}_1,\ \dots, \mathcal{M}_n$ be the remaining conjugacy classes of maximal subgroups of $G$. Define $x_j  = \lvert \mathcal{M}_j \cap \mathcal{C}\rvert$ for $1 \leq j \leq n$. For each remaining conjugacy class $\mathcal{P}$ of principal subgroups of $G$ the following inequality must hold:

\[b_{\mathcal{P},\mathcal{M}_1}x_1 + b_{\mathcal{P},\mathcal{M}_2}x_2+ \dots +b_{\mathcal{P},\mathcal{M}_n}x_n \geq \lvert \mathcal{P} \rvert.\] 

The sum of $\lvert U \rvert$ and the minimum value of the function $x_1 + \dots +x_n$ subject to these linear constraints, as well as the conditions $x_j\in \mathbb{Z}$, and $0 \leq x_j \leq \lvert \mathcal{M}_j \rvert$ for $1\leq j \leq n$, is a lower bound for $\sigma(G)$. If this lower bound is equal to the number of subgroups in a cover $\mathcal{C}$, then necessarily $\mathcal{C}$ is minimal and $\sigma(G) = \lvert \mathcal{C} \rvert$, though in general this bound may be strictly less than the actual covering number of the group. One may be able to work around this difficulty in some instances by deriving additional linear constraints by, for example, considering how the group acts on appropriate combinatorial objects. This is illustrated in the computation of the covering number of the McLaughlin group in Section \ref{McL}.

In cases where the lower bound described above is less than the number of subgroups in the best cover we have found, we can formulate a different linear programming problem to find the covering number. Let $M$ be the matrix with rows and columns indexed by the remaining principal subgroups and the remaining maximal subgroups of $G$ respectively, such that the entry in the row $C$ and column $H$ is 1 if $C\leq H$ and 0 otherwise. Then a minimal cover of $G$ consisting of maximal subgroups is then of the form $U \cup X$, where the members of $X$ are the maximal subgroups corresponding to the positions of the 1's in a (0,1)-vector $x$ whose entries have minimum possible sum subject to the condition that $Mx\geq \vec{1}$, where $\vec{1}$ denotes the column vector of all 1's of length equal to the number of rows of $M$. The benefit of this second ILP formulation is that a solution to the linear programming problem always yields the exact covering number of the group, but the trade off is that it generally has far more variables and constraints than the other linear programming problem described above, and is usually much more difficult to solve. However, even if the problem is intractable, we can get upper and lower bounds on the covering number from an incomplete attempt to solve the ILP from the best incumbent solution and best bound found by the solver at the time it is interrupted. We take this approach in Section \ref{Primitive} when investigating the covering numbers of low degree primitive groups.

\section{The McLaughlin Group} \label{McL}

The McLaughlin group, one of the sporadic simple groups, is a subgroup of index two in the full automorphism group of the McLaughlin graph, a strongly regular graph with parameters $(275,112,30,56)$. We note that the independence number of the McLaughlin graph is 22 (see \cite{Brouwer}), and that the McLaughlin group acts transitively on the vertices, the edges, and the nonedges of the McLaughlin graph. The maximal subgroups from classes $\mathcal{M}_1$, $\mathcal{M}_6$, and $\mathcal{M}_7$ are the stabilizers of the vertices, edges, and nonedges of the McLaughlin graph respectively. Tables \ref{tab:McLPCS} and \ref{tab:McLmax} give the conjugacy classes of principal and maximal subgroups of the McLaughlin group, and the matrix $A$ for the McLaughlin group is given in Table \ref{tab:McLincidence}.

\begin{table}[H]
\caption{Conjugacy classes of principal subgroups of the McLaughlin group}
\label{tab:McLPCS}
\begin{center}
\begin{tabular}{crr}
Class & Order & Class Size\\
$\mathcal{P}_1$ & 5 & 8981280\\
$\mathcal{P}_2$ & 6 & 12474000\\
$\mathcal{P}_3$ & 8 & 28066500\\
$\mathcal{P}_4$ & 9 & 11088000\\
$\mathcal{P}_5$ & 11 & 16329600\\
$\mathcal{P}_6$ & 12 & 18711000\\
$\mathcal{P}_7$ & 14 & 21384000\\
$\mathcal{P}_8$ & 30 & 7484400\\
\end{tabular}
\end{center}
\end{table}

\begin{table}[H]
\caption{Conjugacy classes of maximal subgroups of the McLaughlin group}
\label{tab:McLmax}
\begin{center}
\begin{tabular}{crrc}
Class & Order & Class Size & Structure\\
$\mathcal{M}_1$ & 3265920 & 275 & $U_4(3)$\\
$\mathcal{M}_2$ & 443520 & 2025 & $M_{22}$\\
$\mathcal{M}_3$ & 443520 & 2025 & $M_{22}$\\
$\mathcal{M}_4$ & 126000 & 7128 & $U_3(5)$\\
$\mathcal{M}_5$ & 58320 & 15400 & $3_+^{1+4}:2S_5$\\
$\mathcal{M}_6$ & 58320 & 15400 & $3^4:M_{10}$\\
$\mathcal{M}_7$ & 40320 & 22275 & $L_3(4):2$\\
$\mathcal{M}_8$ & 40320 & 22275 & $2.A_8$\\
$\mathcal{M}_9$ & 40320 & 22275 & $2^4:A_7$\\
$\mathcal{M}_{10}$ & 40320 & 22275 & $2^4:A_7$\\
$\mathcal{M}_{11}$ & 7920 & 113400 & $M_{11}$\\
$\mathcal{M}_{12}$ & 3000 & 299376 & $5_+^{1+2}:3:8$\\
\end{tabular}
\end{center}
\end{table}

\setlength{\tabcolsep}{4pt}
\begin{table}[H]
\caption{Matrix $A$ for $McL$}
\label{tab:McLincidence}
\begin{center}
\begin{tabular}{c|cccccccccccc}
& $\mathcal{M}_1$ & $\mathcal{M}_2$ & $\mathcal{M}_3$ & $\mathcal{M}_4$ & $\mathcal{M}_5$ & $\mathcal{M}_6$ & $\mathcal{M}_7$ & $\mathcal{M}_8$ & $\mathcal{M}_9$ & $\mathcal{M}_{10}$ & $\mathcal{M}_{11}$ & $\mathcal{M}_{12}$\\ \hline
$\mathcal{P}_1$ & 5 & 5 & 5 & 3 & 0 & 5 & 5 & 0 & 5 & 5 & 5 & 1\\
$\mathcal{P}_2$ & 2 & 3 & 3 & 3 & 2 & 1 & 6 & 7 & 6 & 6 & 6 & 0\\
$\mathcal{P}_3$ & 1 & 1 & 1 & 2 & 2 & 2 & 1 & 1 & 1 & 1 & 2 & 4 \\
$\mathcal{P}_4$ & 2 & 0 & 0 & 0 & 1 & 1 & 0 & 0 & 0 & 0 & 0 & 0\\
$\mathcal{P}_5$ & 0 & 1 & 1 & 0 & 0 & 0 & 0 & 0 & 0 & 0 & 1 & 0\\
$\mathcal{P}_6$ & 1 & 0 & 0 & 0 & 3 & 2 & 0 & 1 & 0 & 0 & 0 & 2\\
$\mathcal{P}_7$ & 0 & 0 & 0 & 0 & 0 & 0 & 1 & 1 & 1 & 1 & 0 & 0 \\
$\mathcal{P}_8$ & 0 & 0 & 0 & 0 & 1 & 0 & 0 & 1 & 0 & 0 & 0 & 1\\
\end{tabular}
\end{center}
\end{table}

\begin{thm}
The covering number of the McLaughlin group is 24553.
\end{thm}

\begin{proof}
We show that there exists a cover $\mathcal{C} \subseteq \mathcal{M}_1 \cup \mathcal{M}_2 \cup \mathcal{M}_8$ with $\lvert \mathcal{C} \rvert = 24553$ as follows: First, observe that $\mathcal{M}_2 \cup \mathcal{M}_8$ is sufficient to cover all of the principal elements of $McL$ except for those of order 9, which generate the cyclic subgroups from $\mathcal{P}_4$, and that these are contained in the subgroups from class $\mathcal{M}_1$. Each element of order 9 fixes two adjacent vertices in the McLaughlin graph, the edge between them, and no nonedges. As the independence number of the McLaughlin graph is 22, we may choose an independent set $I$ consisting of 22 vertices of the McLaughlin graph. Let $S$ be the subset of $\mathcal{M}_1$ consisting of the stabilizers of the vertices of the McLaughlin graph which are not in $I$. Then $\lvert S \rvert = 253$, and given any element $x$ of order 9 in $McL$, at least one of the two subgroups from $\mathcal{M}_1$ which contain $x$ is in $S$. Consequently, $\mathcal{C} = S \cup \mathcal{M}_2 \cup \mathcal{M}_8$ is a cover of $McL$ with $\lvert \mathcal{C} \rvert = 253 +2025 +22275 = 24553$. As a result, $\sigma(McL) \leq 24553$.

We now prove that no smaller cover can exist. Suppose that $\mathcal{C}$ is a minimal cover of $McL$ consisting of maximal subgroups. For $1 \leq j \leq 12$, let $x_j = \lvert \mathcal{C} \cap \mathcal{M}_j \rvert$. The cyclic subgroups from $\mathcal{P}_5$ of order 11 are contained only in the maximal subgroups from classes $\mathcal{M}_2$, $\mathcal{M}_3$, and $\mathcal{M}_{11}$. Each subgroup $H \in \mathcal{M}_2 \cup \mathcal{M}_3$ contains 8064 of these cyclic subgroups and each subgroup $H \in \mathcal{M}_{11}$ contains 144 of them. Since $\mathcal{C}$ covers all of the elements of order 11,
\[8064(x_2 +x_3) +144x_{11} \geq \lvert \mathcal{P}_5 \rvert = 16329600, \]
from which it follows that $x_2 +x_3 +x_{11} \geq 2025$. Now, the cyclic subgroups of order 14 from $\mathcal{P}_7$ are contained only within the maximal subgroups from classes $\mathcal{M}_i$ for $7 \leq i \leq 10$, and a subgroup from any of these classes contains exactly 960 members of $\mathcal{P}_7$. Hence,
\[ 960(x_7 +x_8 + x_9 + x_{10}) \geq \lvert \mathcal{P}_7 \rvert = 21384000, \]
from which we may deduce that $x_7 +x_8 + x_9 + x_{10} \geq 22275$. A similar analysis for the subgroups from $\mathcal{P}_4$ will yield the inequality $x_1 +x_5 +x_6 \geq 138$, but this is insufficient for our purposes. We consider the action of the principal subgroups of $McL$ of order 9 on the McLaughlin graph to obtain a stronger inequality. Let $W$ be the set of vertices of the McLaughlin graph whose stabilizers are \emph{not} contained in $\mathcal{C}$. Observe that $x_1 + \lvert W \rvert = 275$. Now, an cyclic subgroup $\langle x\rangle$ of order 9 in $McL$ is left uncovered by the subgroups from $\mathcal{M}_1 \cap \mathcal{C}$ if and only if the unique edge $e$ of the McLaughlin graph fixed by $x$ is an edge of the subgraph induced by $W$. The stabilizer in $McL$ of an edge of the McLaughlin graph contains 720 cyclic subgroups of order 9, so the total number of members of $\mathcal{P}_4$ left uncovered by the subgroups from $\mathcal{M}_1 \cap \mathcal{C}$ is $720n$, where $n$ is the number of edges in the subgraph of the McLaughlin graph induced by $W$. Aside from the subgroups from class $\mathcal{M}_1$, the only maximal subgroups which contain members of $\mathcal{P}_4$ are those from $\mathcal{M}_5 \cup \mathcal{M}_6$, each of which contains 720 of them. Consequently, at least $n$ subgroups from $\mathcal{M}_5 \cup \mathcal{M}_6$ are needed to cover the remaining cyclic subgroups of order 9. Now, the McLaughlin graph has independence number 22, so the subgraph induced by $W$ can have no independent set of more than 22 vertices. Therefore $n \geq \lvert W \rvert - 22$. Thus,
\[x_1 +x_5 + x_6 \geq x_1 +n \geq x_1 + \lvert W \rvert -22 = 275 - 22 = 253.\]
It follows from the inequalities $x_2 +x_3 +x_{11} \geq 2025$, $x_7 +x_8 + x_9 + x_{10} \geq 22275$, and $x_1 +x_5 + x_6 \geq 253$ that $\sigma(McL) = \lvert \mathcal{C} \rvert = \sum_{j=1}^{12} x_j \geq 2025 +22275 +253 = 24553$, which completes the proof.
\end{proof}

This settles one of two open cases for which upper and lower bounds were given in \cite{Holmes2006}. The other such case, that of the Janko group $J_1$, remains open, though improved bounds were given in \cite{Kappe2016}.

\section{Some Primitive Groups of Low Degree} \label{Primitive}
In this section we take a computational approach to investigating the covering numbers of some of the remaining primitive groups of degree less than 129 from Table 2 of \cite{Garonzi2019}. The approach is based on linear programming and follows the method described in Section \ref{Prelim}. All of the group-theoretic computations were done using Magma \cite{Magma}, and the linear programming problems were solved using Gurobi \cite{Gurobi}. In many cases we are unable to solve the larger ILP problem in a reasonable amount of time; in this case we can still obtain upper and lower bounds for the covering number from Gurobi's best incumbent solution and best lower bound at the time the computation is interrupted. We summarize our results in Tables \ref{tab:coveringnumbers} and \ref{tab:bounds}.

\begin{table}[H]
    \centering
    \begin{tabular}{c|cc|c}
      & Previous & Previous &\\
     Group & Lower Bound & Upper bound & Covering Number\\
      \hline
      $L_5(3)$ & 393030144 & -- & 393031475\\
      $P\Sigma L_2(121)$ & 671 & 794 & 794\\
       $A_5\ \text{wr}\ 3$ & 216 & 342 & 317\\
       $(A_5 \times A_5 \times A_5).6$ & 1000 & 1217 & 1127\\
       $P\Gamma L_2(125)$ & 7750 & 7876 & 7876\\
       $L_7(2)$ & 184308203520 & -- & 184308218125
    \end{tabular}
    \caption{Exact covering numbers of some of the low degree primitive groups from the list of open cases in \cite{Garonzi2019}.}
    \label{tab:coveringnumbers}
\end{table}

\begin{table}[H]
    \centering
    \begin{tabular}{c|cc|cc}
      & Previous & Previous & New & New\\
     Group & Lower Bound & Upper bound & Lower Bound & Upper Bound\\
      \hline
      $A_7\ \text{wr}\ 2$ & 447 & 667 & 460 & 667\\
      $L_3(4).2_2$ & 138 & 166 & 144 & 166\\
       $HS:2$ & 11859 & 22375 & 15127 & 22375\\
       $L_2(11)\ \text{wr}\ 2$ & 570 & 926 & 721 & 817\\
       $PGU_3(5)$ & 6000 & 6526 & 6307 & 6378
    \end{tabular}
    \caption{Improved upper and lower bounds for the covering numbers of some of the low degree primitive groups from the list of open cases in \cite{Garonzi2019}.}
    \label{tab:bounds}
\end{table}

We illustrate our method by presenting two cases in detail. The others are similar.

\subsection{$L_5(3)$}
$L_3(5)$ has 17 conjugacy classes of principal subgroups and 8 classes of maximal subgroups. These are as in Tables \ref{tab:L53PCS} and \ref{tab:L53max}. We note that each member of $\mathcal{P}_{17}$ is contained in a unique maximal subgroup of $L_5(3)$, which is from $\mathcal{M}_8$, so we must use all 393030144 subgroups from that class in the cover. The matrix $A'$ is given in Table \ref{tab:L53A'}.\\

\begin{table}[H]
\caption{Conjugacy classes of principal subgroups of $L_5(3)$}
\label{tab:L53PCS}
\begin{center}
\begin{tabular}{crr}
Class & Order & Class Size\\
$\mathcal{P}_1$ & 6 & 366949440\\
$\mathcal{P}_2$ & 6 & 366949440\\
$\mathcal{P}_3$ & 6 & 2201696640\\
$\mathcal{P}_4$ & 8 & 1857681540\\
$\mathcal{P}_5$ & 8 & 928840770\\
$\mathcal{P}_6$ & 9 & 489265920\\
$\mathcal{P}_7$ & 12 & 1238454360\\
$\mathcal{P}_8$ & 18 & 733898880\\
$\mathcal{P}_9$ & 24 & 825636240\\
$\mathcal{P}_{10}$ & 24 & 825636240\\
$\mathcal{P}_{11}$ & 24 & 1238454360\\
$\mathcal{P}_{12}$ & 24 & 137606040\\
$\mathcal{P}_{13}$ & 26 & 1524251520\\
$\mathcal{P}_{14}$ & 78 & 508083840\\
$\mathcal{P}_{15}$ & 80 & 743072616\\
$\mathcal{P}_{16}$ & 104 & 381062880\\
$\mathcal{P}_{17}$ & 121 & 393030144\\
\end{tabular}
\end{center}
\end{table}

\begin{table}[H]
\caption{Conjugacy classes of maximal subgroups of $L_5(3)$}
\label{tab:L53max}
\begin{center}
\begin{tabular}{crrc}
Class & Order & Class Size & Structure\\
$\mathcal{M}_1$ & 1965150720 & 121 & $3:(3^3:2).L_4(3).2$\\
$\mathcal{M}_2$ & 1965150720 & 121 & $3:(3^3:2).L_4(3).2$\\
$\mathcal{M}_3$ & 196515072 & 1210 & $3^6.Q_8.S_3.L_3(3)$\\
$\mathcal{M}_4$ & 196515072 & 1210 & $3^6.Q_8.S_3.L_3(3)$\\
$\mathcal{M}_5$ & 51840 & 4586868 & $C_2(3).2$\\
$\mathcal{M}_6$ & 7920 & 30023136 & $M_{11}$\\
$\mathcal{M}_7$ & 7920 & 30023136 & $M_{11}$\\
$\mathcal{M}_8$ & 605 & 393030144 & $121:5$\\
\end{tabular}
\end{center}
\end{table}

\setlength{\tabcolsep}{4pt}
\begin{table}[H]
\caption{Matrix $A'$ for $L_5(3)$}
\label{tab:L53A'}
\begin{center}
\begin{tabular}{c|ccccccc}
& $\mathcal{M}_1$ & $\mathcal{M}_2$ & $\mathcal{M}_3$ & $\mathcal{M}_4$ & $\mathcal{M}_5$ & $\mathcal{M}_6$ & $\mathcal{M}_7$\\ \hline
$\mathcal{P}_1$ & 5 & 5 & 9 & 9 & 0 & 0 & 0\\
$\mathcal{P}_2$ & 5 & 5 & 8 & 8 & 36 & 0 & 0\\
$\mathcal{P}_3$ & 2 & 2 & 3 & 3 & 0 & 9 & 9\\
$\mathcal{P}_4$ & 1 & 1 & 2 & 2 & 0 & 8 & 8\\
$\mathcal{P}_5$ & 1 & 1 & 2 & 2 & 8 & 0 & 0\\
$\mathcal{P}_6$ & 1 & 1 & 1 & 1 & 9 & 0 & 0\\
$\mathcal{P}_7$ & 2 & 2 & 3 & 3 & 0 & 0 & 0\\
$\mathcal{P}_8$ & 2 & 2 & 2 & 2 & 0 & 0 & 0\\
$\mathcal{P}_9$ & 1 & 1 & 2 & 2 & 0 & 0 & 0\\
$\mathcal{P}_{10}$ & 1 & 1 & 1 & 1 & 0 & 0 & 0\\
$\mathcal{P}_{11}$ & 2 & 2 & 3 & 3 & 0 & 0 & 0\\
$\mathcal{P}_{12}$ & 4 & 4 & 5 & 5 & 0 & 0 & 0\\
$\mathcal{P}_{13}$ & 2 & 2 & 1 & 1 & 0 & 0 & 0\\
$\mathcal{P}_{14}$ & 1 & 1 & 1 & 1 & 0 & 0 & 0\\
$\mathcal{P}_{15}$ & 1 & 1 & 0 & 0 & 0 & 0 & 0\\
$\mathcal{P}_{16}$ & 0 & 0 & 1 & 1 & 0 & 0 & 0\\
\end{tabular}
\end{center}
\end{table}

Observe that $\mathcal{C} = \mathcal{M}_1 \cup \mathcal{M}_3 \cup \mathcal{M}_8$ is a cover of size 393031475, and therefore $\sigma(L_5(3)) \leq 393031475$. On the other hand, it is necessary to use all 393030144 subgroups from class $\mathcal{M}_8$ in order to cover the cyclic subgroups of order 121, and we must use at least 121 subgroups from $\mathcal{M}_1 \cup \mathcal{M}_2$ to cover the members of $\mathcal{P}_{15}$ and at least 1210 subgroups from $\mathcal{M}_3 \cup \mathcal{M}_4$ to cover the members of $\mathcal{P}_{16}$. Consequently $\sigma(L_5(3)) = 393031475$.

\subsection{$(A_5\times A_5 \times A_5).6$}
The group $(A_5\times A_5 \times A_5).6$ has 19 conjugacy classes of principal subgroups and 6 classes of maximal subgroups, described in Tables \ref{tab:A53.6PCS} and \ref{tab:A53.6max}. We note that each member of $\mathcal{P}_{17}$ is contained in a unique maximal subgroup, which is from $\mathcal{M}_5$. Therefore it is necessary to use all 1000 subgroups from $\mathcal{M}_5$ in the cover. Likewise, we must use the subgroup from  $\mathcal{M}_1$ in order to cover the members of $\mathcal{P}_8 \cup \mathcal{P}_9 \cup \mathcal{P}_{10} \cup \mathcal{P}_{11}$. The only classes of principal subgroups whose members are not covered by $\mathcal{M}_1 \cup \mathcal{M}_5$ are $\mathcal{P}_{12}$ and $\mathcal{P}_{15}$. We present the matrix $A'$ in Table \ref{tab:A53.6A'}.\\

\begin{table}[H]
\caption{Conjugacy classes of principal subgroups of $(A_5\times A_5 \times A_5).6$}
\label{tab:A53.6PCS}
\begin{center}
\begin{tabular}{crr}
Class & Order & Class Size\\
$\mathcal{P}_1$ & 5 & 2592\\
$\mathcal{P}_2$ & 6 & 3000\\
$\mathcal{P}_3$ & 6 & 6000\\
$\mathcal{P}_4$ & 6 & 36000\\
$\mathcal{P}_5$ & 10 & 3240\\
$\mathcal{P}_6$ & 10 & 3240\\
$\mathcal{P}_7$ & 10 & 4050\\
$\mathcal{P}_8$ & 12 & 4500\\
$\mathcal{P}_9$ & 12 & 4500\\
$\mathcal{P}_{10}$ & 12 & 9000\\
$\mathcal{P}_{11}$ & 12 & 13500\\
$\mathcal{P}_{12}$ & 12 & 54000\\
$\mathcal{P}_{13}$ & 15 & 2160\\
$\mathcal{P}_{14}$ & 15 & 3600\\
$\mathcal{P}_{15}$ & 15 & 21600\\
$\mathcal{P}_{16}$ & 15 & 2160\\
$\mathcal{P}_{17}$ & 18 & 24000\\
$\mathcal{P}_{18}$ & 30 & 2700\\
$\mathcal{P}_{19}$ & 30 & 2700\\
\end{tabular}
\end{center}
\end{table}

\begin{table}[H]
\caption{Conjugacy classes of maximal subgroups of $(A_5\times A_5 \times A_5).6$}
\label{tab:A53.6max}
\begin{center}
\begin{tabular}{crrc}
Class & Order & Class Size & Structure\\
$\mathcal{M}_1$ & 432000 & 1 & $A_5.A_5.S_5$\\
$\mathcal{M}_2$ & 648000 & 1 & $A_5\ \text{wr} \ 3$\\
$\mathcal{M}_3$ & 10368 & 125 & $2^6.3^3.6$\\
$\mathcal{M}_4$ & 6000 & 216 & $5^3.A_4.4$\\
$\mathcal{M}_5$ & 1296 & 1000 & $3^3.(2^2 \times A_4)$\\
$\mathcal{M}_6$ & 360 & 3600 & $3\times S_5$\\
\end{tabular}
\end{center}
\end{table}

\setlength{\tabcolsep}{4pt}
\begin{table}[H]
\caption{Matrix $A'$ for $(A_5\times A_5 \times A_5).6$}
\label{tab:A53.6A'}
\begin{center}
\begin{tabular}{c|cccc}
& $\mathcal{M}_2$ & $\mathcal{M}_3$ & $\mathcal{M}_4$ & $\mathcal{M}_6$\\ \hline
$\mathcal{P}_{12}$ & 0 & 1 & 2 & 1\\
$\mathcal{P}_{15}$ & 1 & 0 & 1 & 1\\
\end{tabular}
\end{center}
\end{table}

We can cover the remaining principal subgroups with the 126 subgroups from $\mathcal{M}_2 \cup \mathcal{M}_3$. Moreover, covering the members of $\mathcal{P}_{15}$ without using the subgroup from $\mathcal{M}_2$ requires at least 216 subgroups from $\mathcal{M}_4 \cup \mathcal{M}_6$, and will not yield a minimal cover. Therefore we must use the subgroup from $\mathcal{M}_2$, and we can complete a minimal cover by finding a minimal subset of $\mathcal{M}_3 \cup \mathcal{M}_4 \cup \mathcal{M}_6$ which covers the elements of $\mathcal{P}_{12}$. Therefore we construct the $54000 \times 3941$ (0,1)-incidence matrix between the members of $\mathcal{P}_{12}$ and those of $\mathcal{M}_3 \cup \mathcal{M}_4 \cup \mathcal{M}_6$. We use Gurobi to solve the corresponding linear programming problem, and find that the minimum number of subgroups required to cover the members of $\mathcal{P}_{12}$ is indeed 125. Therefore $\mathcal{M}_1 \cup \mathcal{M}_2 \cup \mathcal{M}_3 \cup \mathcal{M}_5$ is a minimal cover of $(A_5\times A_5 \times A_5).6$, and $\sigma((A_5\times A_5 \times A_5).6) = 1127$.

\end{document}